\documentclass{article}

\usepackage{fullpage}
\usepackage{amsmath,amssymb,amsthm,bm}
\usepackage{tikz}
\usepackage{graphicx}
\usepackage{caption,subcaption,float}
\usepackage{epstopdf}
\usepackage{enumerate}
\usepackage[hang,flushmargin]{footmisc}
\usepackage[normalem]{ulem}

\usetikzlibrary{shapes,arrows,spy,positioning,decorations,calc,patterns}

\newtheorem{thm}{Theorem}[section]
\newtheorem{lem}[thm]{Lemma}
\newtheorem{cor}[thm]{Corollary}

\newtheorem{claim}[thm]{Claim}

\title{\Ifps{} of Sparse Graphs}

\author{
Axel Brandt$^{1,6}$
\and
Michael Ferrara$^{1,6,7}$
\and
Mohit Kumbhat$^{2,6}$
\and
Sarah Loeb$^{3,6}$
\and
Derrick Stolee$^{2,4,6}$
\and
Matthew Yancey$^{5,6}$
}

\begin{document}

\newcommand{\Mad}{\operatorname{Mad}}
\newcommand{\p}{\rho}
\newcommand{\f}{\mathcal{F}}
\newcommand{\I}{\mathcal{I}}
\newcommand{\ic}{\mu_0}
\newcommand{\mc}{\mu_1}
\newcommand{\fc}{\mu_3}
\renewcommand{\P}{\varphi}
\renewcommand{\bar}{\overline}
\newcommand{\Ifp}{I,F-partition}
\newcommand{\Ifps}{I,F-partitions}

\maketitle

\footnotetext[1]{Department of Mathematical and Statistical Sciences, University of Colorado Denver, Denver, CO 80217 ; {\tt $\{$axel.brandt,michael.ferrara$\}$@ucdenver.edu}.}
\footnotetext[2]{Department of Mathematics, Iowa State University, Ames, IA, U.S.A. \texttt{$\{$mkumbhat,dstolee$\}$@iastate.edu}}
\footnotetext[3]{Department of Mathematics, University of Illinois at Urbana-Champaign, Urbana, IL, U.S.A. \texttt{sloeb2@illinois.edu}}
\footnotetext[4]{Department of Computer Science, Iowa State University, Ames, IA, U.S.A.}
\footnotetext[5]{Institute for Defense Analyses / Center for Computing Sciences, \texttt{mpyance@super.org}}
\footnotetext[6]{Research supported in part by NSF grant DMS-1500662 ``The 2015 Rocky Mountain - Great Plains Graduate Research Workshop in Combinatorics".}
\footnotetext[7]{Research supported in part by a Collaboration Grant from the Simons Foundation (\#206692 to Michael Ferrara).}

\begin{abstract}
A \textit{star $k$-coloring} is a proper $k$-coloring where the union of two color classes induces a star forest.
While every planar graph is 4-colorable, not every planar graph is star 4-colorable.
One method to produce a star 4-coloring is to partition the vertex set into a 2-independent set and a forest; such a partition is called an \textit{I,F-partition}.
We use a combination of potential functions and discharging to prove that every graph with maximum average degree less than $\frac{5}{2}$ has an I,F-partition, which is sharp and answers a question of Cranston and West [A guide to the discharging method, arXiv:1306.4434].  This result implies that planar graphs of girth at least 10 are star 4-colorable, improving upon previous results of Bu, Cranston, Montassier, Raspaud, and Wang [Star coloring of sparse graphs, {\it J. Graph Theory} \textbf{62} (2009), 201-219].

\noindent
\textbf{Keywords:} star coloring; \Ifp{}; discharging method; potential function

\noindent
\textbf{MSC Code:} 05C10 (planar), 05C15 (coloring), 05C78 (labeling)
\end{abstract}

\section{Introduction}
All graphs considered in this paper are simple, and the reader is refered to \cite{W} for any undefined terminology.  A $k$-coloring $c:V(G) \to \{1,\ldots,k\}$ of a graph $G$ is \textit{proper} if $c$ assigns distinct colors to adjacent vertices. The \textit{chromatic number} of $G$ is the minimum $k$ such that $G$ has a proper $k$-coloring.
First introduced by Gr{\"u}nbaum~\cite{G}, a proper vertex coloring is \emph{acyclic} if the union of any two color classes induces a forest.  The minimum $k$ such that $G$ has an acyclic $k$-coloring is the \emph{acyclic chromatic number} of $G$, denoted $\chi_a(G)$.  An acyclic $k$-coloring of $G$ is a \emph{star $k$-coloring} if the components of the forest induced by the union of two color classes are stars and the minimum $k$ such that $G$ has a star $k$-coloring is the star chromatic number of $G$, denoted $\chi_s(G)$. It follows immediately that $\chi(G) \le \chi_a(G) \le \chi_s(G)$ for any graph $G$, although is not difficult to see that $\chi\ne\chi_a$ in general by considering, for instance, any bipartite graph containing a cycle.  We refer the reader to the thorough survey of Borodin~\cite{BorodinSurvey} for additional results on acyclic and star colorings beyond what we present next.

In this paper, we are interested in the problem of star-coloring planar graphs.
The well-known Four Color Theorem of Appel and Haken \cite{AH,AHK} states that $\chi(G)\leq4$ if $G$ is planar, while Gr{\"u}nbaum~\cite{G} constructed a planar graph with no acyclic 4-coloring (and so, in particular, no star coloring).  Subsequently, Borodin~\cite{B} showed $\chi_a(G)\leq5$ for all planar $G$.
Albertson, Chappell, Kierstead, K{\"u}ndgen, and Ramamurthi~\cite{ACKKR} showed that every planar graph $G$ satisfies $\chi_s(G)\le 20$ and also constructed a planar graph with star chromatic number at least 10.  K\"{u}ndgen and Timmons \cite{KT} proved that every planar graph of girth 6 (respectively 7 and 8) can be star-colored with 8 (respectively 7 and 6) colors.  Kierstead, K\"undgen and Timmons \cite{KKT} showed that every bipartitie planar graph can be star 14-colored, and constructed a bipartite planar graph with star chromatic number 8.  It is worthwhile to note that, while not our focus here, the results in \cite{KT} and \cite{KKT} hold for the natural extension of star-colorings to a list coloring framework.

Given the Four Color Theorem, it is natural to search for conditions that ensure a planar graph can be star 4-colored.
Albertson \textit{et al.}~\cite{ACKKR} also showed that for every girth $g$, there exists a graph $G_g$ with girth at least $g$ and $\chi_s(G_g) = 4$, and further that there is some girth $g$ such that every planar graph of girth at least $g$ is star 4-colorable.
Timmons~\cite{T} showed that $g=14$ is sufficient and also gave a planar graph with girth 7 and star chromatic number 5.
Bu, Cranston, Montassier, Raspaud, and Wang~\cite{BCMRW} improved upon Timmons' result by showing that every planar graph with girth $g\ge 13$ has a star 4-coloring.

 The \emph{maximum average degree} of a graph $G$, denoted $\Mad(G)$, is $\max\limits_{H \subseteq G} \frac{2 |E(H)|}{|V(H)|}$.  The main result of this paper is the following.

\begin{thm} \label{thm:starcol}
If $G$ is a graph with $\Mad(G) < \frac{5}{2}$, then $\chi_s(G) \le 4$.
\end{thm}

A straightforward application of Euler's formula shows that if $G$ is a planar graph with girth at least $g$, then $\Mad(G) < \frac{2 g}{g - 2}$.
Thus, as a corollary to Theorem~\ref{thm:starcol} we have the following improvement on \cite{BCMRW}.

\begin{cor} \label{cor:planar}
If $G$ is a planar graph with girth at least $10$, then $\chi_s(G) \le 4$.
\end{cor}

To prove Theorem~\ref{thm:starcol} we will use \Ifps, which were first introduced in \cite{ACKKR}.
A \emph{2-independent set} in $G$ is a set of vertices that have pairwise distance greater than 2.
An \emph{\Ifp} of a graph $G$ is a partition of $V(G)$ as $\I \sqcup \f$ where $\I$ is a 2-independent set in $G$ and $G[\f]$ is a forest.
Albertson \textit{et al.}~\cite{ACKKR} observed that if $G$ has an \Ifp{} $\I \sqcup \f$, then $\chi_s(G) \le 4$ because $\chi_s(T) \le 3$ for any tree $T\subseteq G[\f]$ and this 3-coloring of $G[\f]$ can be extended to all of $G$ by assigning the vertices in $\I$ a new color.  Note that the converse does not hold; for example, $\chi_s(K_{3,3}) = 4$, but $K_{3,3}$ has no $I,F$-partition.
Timmons~\cite{T} and Bu \textit{et al.}~\cite{BCMRW} showed that maximum average degree less than $\frac{7}{3}$ and $\frac{26}{11}$, respectively, imply the existence of an \Ifp, which in turn imply the abovementioned girth bounds sufficient for a planar graph to be star 4-colorable.
Along the same lines, Theorem~\ref{thm:starcol} is a consequence of the following theorem.

\begin{thm} \label{thm:mad} If $G$ is a graph with $\Mad(G) < \frac{5}{2}$, then $G$ has an \Ifp.
\end{thm}

Theorem~\ref{thm:mad} is sharp in the sense that there are graphs with maximum average degree $\frac{5}{2}$ that do not have an \Ifp.  Indeed, given a cycle $C$, for each vertex $v$ in the cycle add a 3-cycle $a_v b_v c_v$ and the edge $v a_v$. To see that such a graph, which has maximum average degree $\frac{5}{2}$, does not have an \Ifp, simply note that no vertex $v$ on the cycle $C$ can be in the 2-independent set, as then $a_vb_vc_v$ would necessarily have to be in the forest $\f$, an impossibility.  However, this then implies that every vertex on $C$ must be in $\f$, which is also impossible.  Theorem \ref{thm:mad} therefore answers a question of Cranston and West~\cite[Problem 7.11]{CW}, the relevant part of which asks for the maximum bound on maximum average degree that guarantees an \Ifp.

To prove this result, we use the method of potentials as utilized by Kostochka and Yancey~\cite{KY,KY2}, Borodin, Kostochka and Yancey~\cite{BKY}, and Chen, Kim, Kostochka, West and Zhu~\cite{CKKWZ}.
We also strengthen the problem of finding an {\Ifp} by allowing some vertices to be initially assigned to $\I$ and $\f$, and modify the condition about maximum average degree to account for the preassigned vertices.

Going forward, if $X_1,\dots, X_t$ are a partition of a set $X$, then we will write $X=X_1\sqcup\dots\sqcup X_t$.  If $G$ is a graph with $V(G) = I \sqcup F \sqcup U$, we say that $G$ is an \emph{assigned graph}. If $H$ is a subgraph of an assigned graph $G$, then let $I(H) = I \cap V(H)$, $F(H) = F \cap V(H)$, and $U(H) = U \cap V(H)$.
Additionally, the \emph{potential} of $H$ in $G$, denoted $\p_G(H)$, is
\[
\p_G(H) = |I(H)| + 4|F(H)| + 5|U(H)| - 4|E(H)|.
\]
When the context is clear, we omit the subscript, and we use $\p(S)$ to mean $\p(G[S])$ when $S\subseteq V(G)$.

\begin{thm} \label{thm:pot}
Let $G$ be an assigned graph with vertex set partitioned as $I \sqcup F \sqcup U$.
If $\p_G(H) > 0$ for all nonempty subgraphs $H \subseteq G$, then $G$ has an \Ifp{} $\I \sqcup \f$ such that $I \subseteq \I$ and $F \subseteq \f$.
\end{thm}


Note that if $G$ is an assigned graph where two vertices in $I$ are adjacent or have a common neighbor, then there exists a subgraph with nonpositive potential.
Similarly, a cycle of vertices in $F$ form a subgraph of nonpositive potential.
Neither structure therefore appears as a subgraph of any graph satisfying the hypotheses of Theorem~\ref{thm:pot}.
Additionally, adding edges between vertices in a subgraph of $G$ only decreases the potential.
Thus we need only consider induced subgraphs when minimizing the potential across all subgraphs of $G$.

In Section~\ref{sec:equiv}, we demonstrate that Theorem~\ref{thm:mad} and Theorem~\ref{thm:pot} are equivalent.  Hence
Theorems~\ref{thm:mad} and \ref{thm:starcol} and Corollary~\ref{cor:planar} all follow from the proof of Theorem~\ref{thm:pot}, which appears in Section \ref{sec:main}. 
Section~\ref{sec:lemmas} contains the lemmas we will use to proof Theorem~\ref{thm:pot}.
We conclude this section with some further notation.

A $k$-vertex is a vertex of degree $k$ and a $k^+$-vertex is a vertex of degree at least $k$.
For a vertex $v$, $N(v)$ is the neighborhood of $v$.  We reserve $I,F,$ and $U$ as sets of the vertex partition of an assigned graph.  For vertex sets, an overbar indicates the vertex complement, for example $\bar{I}$ is $F \cup U$.

An \Ifp{} $\I\sqcup\f$ \emph{extends} an assignment $I \sqcup F \sqcup U$ if $I \subseteq \I$ and $F \subseteq \f$. For an assigned graph $H$, we say $H$ has an \Ifp{} only if $H$ has an \Ifp{} that extends $I \sqcup F \sqcup U$.
For an \Ifp{} of $H$, let $H_{\f}$ be the subgraph of $H$ induced by vertices assigned to $\f$ and let $H_{\I}$ be the subgraph of $H$ induced by vertices assigned to $\I$.


\section{Proof That Theorem~\ref{thm:mad} and Theorem~\ref{thm:pot} are Equivalent} \label{sec:equiv}

In this section, we demonstrate that Theorems~\ref{thm:mad} and \ref{thm:pot} are equivalent, and in the process demonstrate some of the rationale that led to our definition of the potential function $\rho$.
Note that  if $G$ is an assigned graph and $H \subseteq G$, then straightforward arithmetic shows that $\p(H) > 0$ is equivalent to
\[
\frac{2|E(H)|+22|I(H)|+8|F(H)|}{|U(H)|+9|I(H)|+4|F(H)|} < \frac{5}{2}.
\]
Therefore, Theorem~\ref{thm:pot} implies Theorem~\ref{thm:mad} by taking the sets $I$ and $F$ to be empty.

To establish the converse, let us assume $G$ is an assigned graph with the vertex partition $I \sqcup F \sqcup U$ and $\min\limits_{\emptyset \neq H\subseteq G}\rho(H) > 0$, hence
\[
\max\limits_{\emptyset \neq H \subseteq G} \frac{2|E(H)|+22|I(H)|+8|F(H)|}{|U(H)|+9|I(H)|+4|F(H)|} < \frac{5}{2}.
\]

Let $F = \{ u_1,\dots,u_k\}$ and $I = \{ v_{k+1},\dots, v_\ell \}$.
Starting with $G_0 = G$, iteratively build an auxiliary assigned graph $G_i$ for $i \in \{1,\dots,\ell\}$.
If $i \in \{1,\dots,k\}$, then build $G_i$ by adding a 3-cycle $abc$ and the edge $u_ia$ to $G_{i-1}$, removing $u_i$ from $F$, and adding $u_i, a, b, c$ to $U$; we call this an \emph{$F$-gadget} (see Figure~\ref{fig:fGad}).
If $i \in \{k+1,\dots,\ell\}$, then build $G_i$ by adding 3-cycles $abc$ and $fgh$, a path $adef$, and the edges $v_id$ and $v_ie$ to $G_{i-1}$, removing $v_i$ from $I$, and adding $v_i, a, b, c, d, e, f, g, h$ to $U$; we call this an \emph{$I$-gadget} (see Figure~\ref{fig:iGad}).
Note that the final graph $G_\ell$ has $U(G_\ell) = V(G_\ell)$ and, consequently, potential function $\p_{G_\ell}(H) = 5|V(H)|-4|E(H)|$ for all subgraphs $H \subseteq G_\ell$.

\begin{figure}[!htb]
\centering
\begin{subfigure}[b]{.4\textwidth}
\centering
\begin{tikzpicture}[scale=.5]
\tikzstyle{every node}=[draw,circle,fill=black,minimum size=3pt,
                            inner sep=0pt]
\node[label=left:$u_i$] (1) at (0,0) {};
\node[label=below left:$a$] (2) at (0,2) {};
\node[label=left:$b$] (3) at (-1,3) {};
\node[label=right:$c$] (4) at (1,3){};
\draw (1) -- (2) (2) -- (3) (2) -- (4) (3) -- (4);
\end{tikzpicture}
\caption{\label{fig:fGad}An $F$-gadget forces $u_i$ to be assigned $\f$ in an \Ifp.}
\end{subfigure}
\quad
\begin{subfigure}[b]{.4\textwidth}
\centering
\begin{tikzpicture}[scale=.5]
\tikzstyle{every node}=[draw,circle,fill=black,minimum size=3pt,
                            inner sep=0pt]
\node [label=left:$v_i$] (5) at (0,0) {};
\node[label=right:$d$](6) at (1,1) {};
\node[label=left:$e$] (7) at (-1,1) {};
\node[label=below right:$a$] (8) at (2,2) {};
\node[label=below left:$f$](9) at (-2,2) {};
\node[label=right:$b$] (10) at (3,3) {};
\node[label=left:$c$] (11) at (1,3) {};
\node[label=left:$g$] (12) at (-3,3) {};
\node[label=right:$h$] (13) at (-1,3) {};
\draw (5) -- (6) -- (7) (5) -- (7) (6) -- (8) (7) -- (9) (8) -- (10) (8) -- (11) (10) -- (11) (12) -- (9) -- (13) -- (12);
\end{tikzpicture}
\caption{\label{fig:iGad}An $I$-gadget forces $v_i$ to be assigned $\I$ in an \Ifp.}
\end{subfigure}
\caption{\label{fig:gadgets}The $F$- and $I$-gadgets.}
\end{figure}
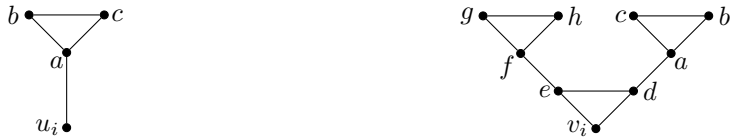

\begin{claim}\label{claim:gadgets}
For all $i \in \{0,1,\dots,\ell\}$ and every subgraph $H \subseteq G_i$, $\p_{G_i}(H) > 0$.
Hence $\Mad(G_\ell) < \frac{5}{2}$.
\end{claim}

\begin{proof}[Proof of Claim~\ref{claim:gadgets}]
We proceed by induction on $i$.  Observe that the case $i = 0$ holds by assumption, and suppose that $0 \leq i < \ell$ and $\p_{G_i}(H) > 0$ for all $H \subseteq G_i$. Select $H' \subseteq G_{i+1}$ such that $\p_{G_{i+1}}(H')$ is minimum.  Note that if $H' \subseteq G_i$ then as all of the elements in each gadget are in $U$,  $\p_{G_{i+1}}(H') \geq \p_{G_i}(H') > 0$.
Hence we may assume $H'\not\subseteq G_i$.

\noindent\textit{Case 1}: $i + 1 \in \{1,\dots, k\}$.
In this case, an $F$-gadget with vertices $a, b, c$ was connected to $u_i$.
Since $H' \not\subseteq G_i$, we have $V(H') \cap \{u_i,a,b,c\} \neq \emptyset$.  We claim that, in fact, the minimality of $\p_{G_{i+1}}(H')$ implies that the entire $F$-gadget added at stage $i+1$ must be contained in $H'$. Let $H''$ be the maximal subgraph of $H'$ contained in $G_i$, which by the induction hypothesis has positive potential in $G_i$.  Note that the subgraph of minimum potential within the triangle $abd$ is infact the triangle itself, which has potential 3.  Thus, taking into account that reassigning $u_{i+1}$ increases the potential of any subgraph containing $u_{i+1}$ by 1, and that the edge $u_{i+1}a$ reduces the potential of any subgraph by 4, we have that \[
\p_{G_{i+1}}(H') \ge \p_{G_i}(H'') + 1 - 4 + \p_{G_i}(abc) = \p_{G_i}(H'')>0.
\]
\smallskip
\noindent\textit{Case 2}: $i + 1 \in \{k+1,\dots, \ell\}$.
In this case, an $I$-gadget with vertices $a, b, c, d, e, f, g, h$ was connected to $v_{i+1}$.  We again claim that the minimality of $\p_{G_{i+1}}(H')$ implies that the entire $I$-gadget added at stage $i+1$ must be contained in $H'$.  This assertion follows by a similarly straightforward analysis to that used in Case 1.
\end{proof}

\noindent By Theorem~\ref{thm:mad} and Claim~\ref{claim:gadgets}, $G_\ell$ has an \Ifp{} $\I \sqcup \f$.
Given an $F$-gadget $\{u_i,a,b,c\}$ in $G_\ell$, one of the vertices $a,b,c$ must be in $\I$, which in turn forces $u_i$ to be in $\f$.  Hence $F \subseteq \f$.  Similarly, for any $I$-gadget $\{v_i,a,\dots,h\}$ in $G_\ell$ one of $a,b$ or $c$ and one of $f,g$ or $h$ must be in $I$, implying that $d$ and $e$ must be in $F$.  Consequently, $v_i\in \I$, as desired.  Hence, we obtain an {\Ifp} of $G$ extending $F\sqcup I\sqcup U$, so that Theorem \ref{thm:mad} implies Theorem \ref{thm:pot}, as desired.

\section{Some Useful Claims}\label{sec:lemmas}

We now begin our proof of Theorem~\ref{thm:pot}.
For the sake of contradiction, suppose there exists an assigned graph $G$ with $V(G) = I \sqcup F\sqcup U$ such that $\p_G(H) > 0$ for all $H \subseteq G$, yet $G$ has no \Ifp{} extending $I \sqcup F \sqcup U$.
Among such counterexamples, select $G$ to minimize $|V(G)| + |E(G)|$.
We use $H' \prec H$ to indicate that $|V(H')| + |E(H')| < |V(H)| + |E(H)|$.

In this section, we use our minimality assumptions to refine the structure of $G$.
The proof of Theorem~\ref{thm:pot} will then be completed in Section~\ref{sec:main} using the discharging method.

An \emph{$\ell$-thread} is a path $P$ of $\ell$ vertices in $U$ that have degree $2$ in $H$ such that the neighbors of the endpoints of $P$ in $H-P$, which we say \emph{border} the thread $P$, are either $3^+$-vertices or are in $I \cup F$.
An $\ell^+$-thread is a thread with at least $\ell$ internal 2-vertices in $U$.
Define an \emph{open} thread to be a thread with 2 bordering vertices and a \emph{closed} thread to be a thread with 1 bordering vertex.
In counting the number of threads incident with a vertex, open threads contribute once to the count and closed threads contribute twice.

We include the proof of the following claim for completeness.

\begin{claim}[Timmons~\cite{T}, Bu \emph{et. al.}\cite{BCMRW}]\label{claim:prevredconfig}
None of the following appear in $G$:
\begin{enumerate}[(C1)]
\item A 1-vertex in $U$.
\item A $3^+$-thread.
\item A $4$-vertex in $U$ incident to four 2-threads.
\end{enumerate}
\end{claim}

\begin{figure}[!htb]
\centering
	\begin{subfigure}[b]{.32\textwidth}
	\centering
	\begin{tikzpicture}
		\draw [white] (-1.5,-1.5) rectangle (1.5,1.5);
		\draw (0,0) -- (1,0);
		\fill (1,0) circle (2pt) node[right] {\scriptsize$v$};
		\draw [black,fill=white] (0,0) circle (2pt);
	\end{tikzpicture}
	\caption{The subgraph described in (C1).}
	\label{fig:c1}
	\end{subfigure}
	\begin{subfigure}[b]{.32\textwidth}
	\centering
	\begin{tikzpicture}
		\draw [white] (-.5,-1.5) rectangle (2.5,1.5);
		\draw (0,0) -- (2,0);
		\foreach \x/\y in {0/0,2/0}
			{\draw [black,fill=white] (\x,\y) circle (2pt);}
		\foreach \x/\y in {.5/0, 1.5/0, 1/0}
			{\fill (\x,\y) circle (2pt);}
		\node[above] at (1,0) {\scriptsize$v$};
		\node[above] at (0,0) {\scriptsize$a$};
		\node[above] at (2,0) {\scriptsize$b$};
	\end{tikzpicture}
	\caption{The subgraph described in (C2).}
	\label{fig:c2}
	\end{subfigure}
	\begin{subfigure}[b]{.32\textwidth}
	\centering
	\begin{tikzpicture}
		\draw (-1.5,0) -- (1.5,0) (0,-1.5) -- (0,1.5);
		\foreach \x/\y in { 0/0, -1/0, -.5/0, .5/0, 1/0, 0/-.5, 0/-1, 0/.5, 0/1}
			{ \fill (\x,\y) circle (2pt); }
		\foreach \x\y in {0/-1.5, -1.5/0, 1.5/0, 0/1.5}
			{\draw [black,fill=white] (\x,\y) circle (2pt);}
		\node[above right] at (0,0) {\scriptsize$v$};
		\end{tikzpicture}
	\caption{A subgraph described in (C3).}
	\label{fig:c3}
	\end{subfigure}
\caption{Subgraphs described in Claim~\ref{claim:prevredconfig} where all solid vertices are in $U$.}
\label{fig:prevredconfig}
\end{figure}
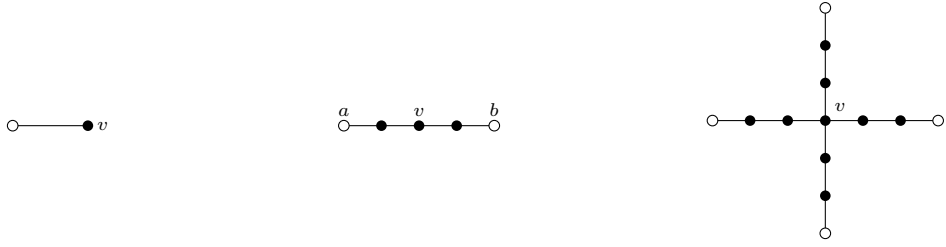
\begin{proof}
Suppose (C1) appears in $G$ as shown in Figure~\ref{fig:c1} where $v \in U$.
Then $G-v \prec G$.
By the minimality of $G$, $G-v$ therefore has an \Ifp{} $\I\sqcup\f$.
Extend $\I\sqcup\f$ to $G$ by assigning $v$ to $\f$.
Doing so does not decrease the distance in $G$ between vertices of $G_{\I}$ and does not create a cycle in $G_{\f}$ since $d(v)=1$.
Thus, this extension is an \Ifp{} of $G$, contradicting the choice of $G$.

Next, suppose (C2) appears in $G$ as shown in Figure~\ref{fig:c2}.
Note that it is possible that $a$ or $b$ are internal to a larger thread containing $v$, or that $a = b$.
Obtain $G'$ from $G$ by deleting $v$ and its neighbors, and note that $G'\prec G$, which implies $G'$ has an \Ifp{} $\I\sqcup\f$ extending $I \sqcup F \sqcup U$.

If at least one of $a$ or $b$ is in $G_{\I}$, then assigning the deleted vertices to $\f$ does not create an $\f$-cycle.
Otherwise, $v$ is at least distance 3 from a vertex in $\I$.
Thus assigning $v$ to $\I$ and the neighbors of $v$ to $\f$ preserves the distance requirement for vertices in $\I$ and does not introduce any $\f$-cycles.
In either case, $\I\sqcup\f$ extends to an \Ifp{} of $G$, again a contradiction.

Finally, assume that (C3) appears in $G$ with 4-vertex $v$.
See Figure~\ref{fig:c3}, although note that we neither assume that the threads incident to $v$ are open, nor that the boundary vertices of these threads are distinct.
The graph $G'$ obtained by deleting $v$ and its incident threads satisfies $G' \prec G$, and once again has an \Ifp{} $\I\sqcup\f$ by the minimality of $G$.
Notice that $v$ is distance at least 3 from any vertex in $\I$, so
assigning $v$ to $\I$ and the other deleted vertices to $\f$ extends $\I\sqcup\f$ to an {\Ifp} of $G$, the final contradiction necessary to complete the claim.
\end{proof}

Before proceeding to our key lemmas, we have the following claims about cut sets in $G$ and the structure of small sets of small potential. For $S\subseteq V(G)$, an \emph{$S$-lobe} of $G$ is an induced subgraph of $G$ whose vertex set consists of $S$ and the vertices of some component of $G-S$.

\begin{claim}\label{claim:cutset}
If $R \subseteq I$, then $G - R$ is connected.
\end{claim}
\begin{proof}
Otherwise, every $R$-lobe $G_i$ is a proper subgraph of $G$ and, by the minimality of $G$, there exists an \Ifp{} $\I_i \sqcup \f_i$ of $G_i$ with $I(G_i) \subseteq \I_i$ and $F(G_i) \subseteq \f_i$.
Consider $\I = \cup \I_i$ and $\f = \cup \f_i$.
Since $R\subseteq I$, $\I$ has no vertices within distance two and $\f$ contains no cycles.
Hence the partition $\I \sqcup \f$ is an \Ifp{} of $G$ extending $I \sqcup F \sqcup U$, a contradiction.
\end{proof}

\begin{claim} \label{claim:trivial}
If $H$ is an induced proper subgraph of $G$ with $\p(H)<3$ and $|V(H)|+|E(H)|\leq4$, then $H$ is one of the five assigned graphs shown in Figure~\ref{fig:trivial}.
\end{claim}
\begin{figure}[!htb]
\centering
	\begin{subfigure}[b]{.19\textwidth}
	\centering
	\begin{tikzpicture}
		\fill (0,0) circle (2pt) node[above] {\scriptsize$I$};
	\end{tikzpicture}
	\caption{$\p(K_1)=1$}
	\end{subfigure}
	\begin{subfigure}[b]{.19\textwidth}
	\centering
	\begin{tikzpicture}
		\fill (0,0) circle (2pt) node[above] {\scriptsize$I$};
		\fill (1,0) circle (2pt) node[above] {\scriptsize$I$};
	\end{tikzpicture}
	\caption{$\p(2K_1)=2$}
	\end{subfigure}
\begin{subfigure}[b]{.19\textwidth}
	\centering
	\begin{tikzpicture}
		\draw (0,0) -- (1,0);
		\fill (0,0) circle (2pt) node[above] {\scriptsize$I$};
		\fill (1,0) circle (2pt) node[above] {\scriptsize$U$};
	\end{tikzpicture}
	\caption{$\p(K_2)=2$}
	\end{subfigure}
	\begin{subfigure}[b]{.19\textwidth}
	\centering
	\begin{tikzpicture}
		\draw (0,0) -- (1,0);
		\fill (0,0) circle (2pt) node[above] {\scriptsize$I$};
		\fill (1,0) circle (2pt) node[above] {\scriptsize$F$};
	\end{tikzpicture}
	\caption{$\p(K_2)=1$}
	\end{subfigure}
	\begin{subfigure}[b]{.19\textwidth}
	\centering
	\begin{tikzpicture}
		\draw (0,0) -- (1,0);
		\fill (0,0) circle (2pt) node[above] {\scriptsize$I$};
		\fill (1,0) circle (2pt) node[above] {\scriptsize$F$};
		\fill (2,0) circle (2pt) node[above] {\scriptsize$I$};
	\end{tikzpicture}
	\caption{$\p(K_2+K_1)=2$}
	\end{subfigure}
\caption{\label{fig:smallconfigs}The five induced proper subgraphs $H$ of $G$ with $\p(H)<3$ and $|V(H)|+|E(H)|\leq4$.}
\label{fig:trivial}
\end{figure}
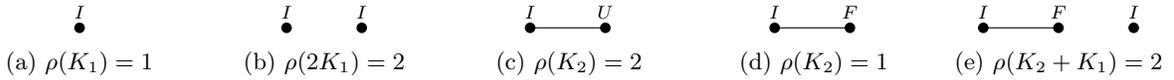


\begin{proof}
Let $H$ be an induced proper subgraph of $G$ with $\p(H)<3$ and $|V(H)|+|E(H)|\leq4$.
If $|E(H)|=2$ then $|V(H)|+|E(H)|\geq5$, so $|E(H)|\in\{0,1\}$.
If $|E(H)|=0$, then as $\p(H)<3$, every vertex must be assigned to $I$ and there are either one or two such vertices as in Figures \ref{fig:smallconfigs}a and \ref{fig:smallconfigs}b.
Otherwise, if $|E(H)|=1$, then the combined potential of the (at least 2) vertices can be at most 6 and must be at least 5 since $\p(H)>0$, hence exactly one vertex of $H$ is in either $U$ or $F$.
Now, if one vertex is in $U$, then these conditions force $H$ to have exactly one other vertex in $I$ as depicted in Figure \ref{fig:smallconfigs}c.
Instead, if one vertex is in $F$, then $H$ has either one or two additional vertices in $I$.
As vertices in $I$ are necessarily nonadjacent, this leaves Figures \ref{fig:smallconfigs}d and \ref{fig:smallconfigs}e as the remaining feasible configurations.
\end{proof}

Lemma~\ref{lem:kahuna}, which we prove next, states that $G$ has no large sets of small potential.
This will be used in the proof of Lemma~\ref{lem:move} which gives hypotheses that allow us to reassign vertices from $U$ to $F$.

\begin{lem} \label{lem:kahuna}
Let $S$ be a proper subset of $G$ and let $H=G[S]$.
If $|S|+|E(H)|\geq5$, then $\p(S)\geq3$.
\end{lem}

\begin{proof}
Suppose otherwise, so there is a proper subset $S$ with $|S| + |E(H)| \geq 5$ and $\p_G(S) < 3$.
Select such $S$ to minimize $\p_G(S)$, and recall that $\p_G(S)>0$.
Further, let $$T=\{v\in S \colon N(v)\cap \bar{S}\neq\emptyset\}.$$

By Claim~\ref{claim:cutset}, if $T \subseteq I$ then $S = T$. 
As vertices in $I$ are pairwise independent, we then would have that $|S|+|E(H)| = |S|\ge 5$ and thus that $\p_G(H)\ge 5$, a contradiction.
Therefore, at least one vertex in  $T$ is in $F\cup U$.

If $\p_G(S)=2$ and $T \cap F = \emptyset$, then modify the assigned graph $H$ to an assigned graph $H_0$ by changing a vertex of $T \cap U$ from $U$ to $F$.
Note that for all $S' \subseteq S$, we have $\p_{H_0}(S') \ge \p_H(S')-5+4=\p_H(S')-1$.
By the minimality of $\p_G(S)$, $\p_H(S') = \p_G(S') \geq 2$.
Thus $\p_{H_0}(S') \geq 1$ for all  $S' \subseteq S$.
Otherwise, let $H_0 = H$, and our assumptions guarantee $\p_{H_0}(S') \geq 1$ for all $S' \subseteq S$.
This alteration, if necessary, will aid in the construction of an auxiliary graph, which we describe below.

By the minimality of $G$, the fact that potentials are minimized by induced subgraphs, and $S \subsetneq V(G)$, there is an \Ifp{} $\I_{H_0} \sqcup \f_{H_0}$ of $H_0$ such that $I(H_0) \subseteq \I_{H_0}$ and $F(H_0) \subseteq \f_{H_0}$.
Let $N_{\I}$ be the set of vertices in $\overline{S}$ adjacent to a vertex $t \in T$ with $t \in \I_{H_0}$ and $N_{\f}$ be the set of vertices in $\overline{S}$ adjacent to a vertex $t \in T$ with $t \in \f_{H_0}$.
Note that $N_\I \cup N_\f \neq \emptyset$ by the definition of $T$.

Construct an auxiliary graph $G'$ by adding vertices to $G-S$ as follows.  If $N_\I \neq \emptyset$, add a new vertex $w$ that is adjacent to every vertex in $N_\I$.  If $N_\f\neq\emptyset$, then add adjacent vertices $x$ and $y$ and connect $y$ to each vertex in $N_{\f}$.
Add $w$ and/or $x$ to $I$, and $y$ to $U$ and let $X$ denote those of $w,x$ and $y$ that are added to $G'$.
See Figure~\ref{fig:G'} for a visual representation of $G'$.

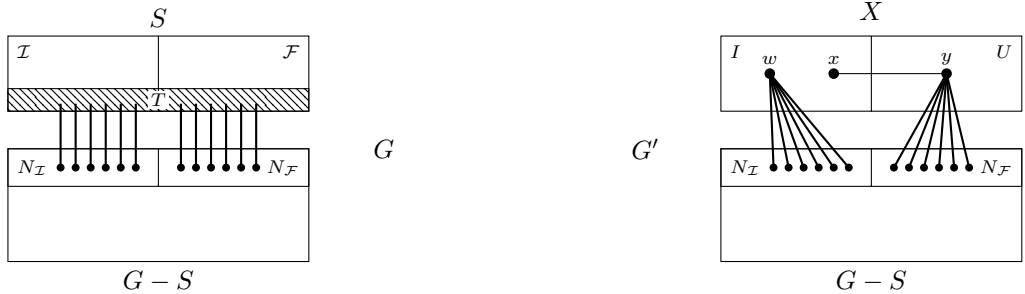
\begin{figure}[!htb]
\centering
	\begin{subfigure}[b]{.49\textwidth}
	\centering
	\begin{tikzpicture}
		\draw (1,1.) -- (1,2)
			(-1,1) rectangle (3,2);
		\node at (1,2)[above] {$S$};
		\draw [pattern=north west lines] (-1,1) rectangle (3,1.3);
		\node[color=black,fill=white,inner sep=1] at (1,1.15){\scriptsize$T$};
		\node at (3,2)[below left] {\scriptsize$\f$};
		\node at (-1,2)[below right] {\scriptsize$\I$};
		
		\draw (-1,.5) rectangle (3,-1);
		\node at (1,-1)[below] {$G-S$};
		\draw (1,0) -- (1,.5) (-1,0) rectangle (3,.5);
		\node at (-1,.25)[right] {\scriptsize$N_{\I}$};
		\node at (3,.25)[left] {\scriptsize$N_{\f}$};
		
		\foreach \x in {-.3,-.1,...,.7,1.3,1.5,...,2.3} {
			\draw[thick] (\x,1.1) -- (\x,.25);
			\fill (\x,.25) circle (1.5pt);}
		
		\node at (4,.5) {$G$};
	\end{tikzpicture}
	\caption{Vertex sets and some adjacencies of $G$.}
	\label{}
	\end{subfigure}
	\begin{subfigure}[b]{.49\textwidth}
	\centering
	\begin{tikzpicture}
		\draw (1,1.) -- (1,2)
			(-1,1) rectangle (3,2);
		\node at (1,2)[above] {{\color{white}[}$X${\color{white}]}};
		\node at (3,2)[below left] {\scriptsize$U$};
		\node at (-1,2)[below right] {\scriptsize$I$};
		
		\draw (-1,.5) rectangle (3,-1);
		\node at (1,-1) [below] {$G-S$};
		\draw (1,0) -- (1,.5) (-1,0) rectangle (3,.5);
		\node at (-1,.25)[right] {\scriptsize$N_{\I}$};
		\node at (3,.25)[left] {\scriptsize$N_{\f}$};
		
		\fill (-.35,1.5) circle (2pt) node[above] {\scriptsize$w$};
		\foreach \x in {-.3,-.1,...,.7}
			{\draw[thick] (-.35,1.5) -- (\x,.25);
			\fill (\x,.25) circle (1.5pt);}
		\fill (.5,1.5) circle (2pt) node[above] {\scriptsize$x$};
		\draw (.5,1.5) -- (2,1.5);
		\fill (2,1.5) circle (2pt) node[above] {\scriptsize$y$};
		\foreach \x in {1.3,1.5,...,2.3}
			{\draw[thick] (2,1.5) -- (\x,.25);
			\fill (\x,.25) circle (1.5pt);}
		
		\node at (-2,.5) {$G'$};
	\end{tikzpicture}
	\caption{Vertex sets and some adjacencies of $G'$.}
	\label{}
	\end{subfigure}
\caption{The construction of $G'$ from $G$ as described in Lemma~\ref{lem:kahuna}.}
\label{fig:G'}
\end{figure}

Observe the following statements about $G'$:
\begin{enumerate}[(Ob 1)]
\addtocounter{enumi}{-1}
\item $|N_G(v)\cap T|\leq1$ for all $v\in \overline{S}$, as otherwise $\p_G(S+v)\leq\p_G(S)+\p_G(v)-4\cdot 2\leq3+5-8= 0$, contradicting the hypothesis on $G$; hence $|N_G(v)\cap S| = |N_{G'}(v)\cap X|$.
\item $X$ is nonempty. 
\item If $\p_{G}(S)=2$, then $T\cap F(H_0) \neq\emptyset$ by construction, and hence $\{x,y\} \subseteq X$.
\end{enumerate}

Since $|X|+|E(G'[X])|\leq4$ and $|S| + |E(H)| \geq 5$, $G' \prec G$.
By the assignment of vertices in $X$ under the construction of $G'$, if $G'$ has an \Ifp{} $\I_{G'} \sqcup \f_{G'}$, then specifically $y \in \f_{G'}$ if $\{x,y\} \subseteq X$.
Observe that $(\I_{G'-X} \cup \I_{H_0})\sqcup (\f_{G'-X} \cup \f_{H_0})$ is an \Ifp{} of $G$ because an $\f$-cycle cannot be formed and the  construction of $X$ implies $\I$ is necessarily a 2-independent set.
Thus, by minimality of $G$, there is instead some $W\subseteq V(G')$ with $\p_{G'}(W) \leq 0$.
Select $W \subseteq V(G')$ to minimize $\p_{G'}(W)$.
Notice that if $W \cap X = \emptyset$, then $W\subset G$ and $\p_G(W) = \p_{G'}(W) \leq 0$, a contradiction.
We may therefore assume $W\cap X\neq \emptyset$.
Observe that $\p_{G'}(W\cap X) \geq 1$.

The minimality of $\p_{G'}(W)$ and the assignment of vertices in $X$ imply that if $W\cap N_{\I}\neq\emptyset$, then $w\in W$, and that if $W\cap N_{\f}\neq\emptyset$, then $\{x,y\}\subseteq W$.
Since every edge between $W\setminus X$ and $X$ in $G'$ corresponds to an edge between $W \setminus S$ and $S$ in $G$ by (Ob 0), we have
\[
0<\p_G(W-X+S)\leq\p_{G'}(W)-\p_{G'}(W\cap X)+\p_{H_0}(S).
\]
Since $\p_{G'}(W)\leq0$, it follows that
\begin{equation}\label{eqn:potW} \p_{G'}(W\cap X)<\p_{H_0}(S). \end{equation}

We have two cases to consider.
First, suppose that $\overline{S}\nsubseteq W$.
Define $S'=W-X+S$ and recall $\p_G(S')  \leq \p_{G'}(W) -\p_{G'}(W\cap X)+\p_G(S)$, which implies $\p_G(S')<\p_{H_0}(S) \leq \p_G(S)$ since $\p_{G'}(W) \leq 0$ and $\p_{G'}(W\cap X) > 0$.
Since $\overline{S}\nsubseteq W$, we have $S'\subsetneq V(G)$, which contradicts the minimality of $\p(S)$.

Now suppose $\overline{S}\subseteq W$.
By the minimality of $\p_{G'}(W)$, $W\cap X=X$ and hence $W = V(G')$.
Thus, $\p_{G'}(G') = \p_{G}(G) - \p_{G}(S) + \p_{G'}(X)$.
Since $\p_{G'}(X) \in \{1,2\}$, $\p_{G}(S) \in \{1,2\}$, $\p_{G}(G) > 0$, and $\p_{G'}(G') \leq 0$, we must have $\p_{G}(S) = 2$, and $\p_{G'}(X) = 1$.
However, by (Ob 2) we have that $\p_{G}(S)=2$ implies $\p_{G'}(X) \geq 2$, a contradiction.
\end{proof}

We can immediately use Lemma~\ref{lem:kahuna} to prove a claim about vertices in $F$.

\begin{claim}\label{claim:2-F}
In $G$, $F$ contains neither a 1-vertex nor a 2-vertex.
\end{claim}

\begin{proof}
Let $v$ be a 1-vertex in $F$, hence $G-v \prec G$.
Also, $\p_{G-v}(H) = \p_G(H) > 0$ for all nonempty subgraphs $H \subseteq G-v$.
Consequently, $G-v$ has an \Ifp{} $\I \sqcup \f$ which can be extended to an \Ifp{} of $G$ by assigning $v$ to $\f$.

If, instead, $v$ is a 2-vertex in $F$, let $u$ be a neighbor of $v$.
If there is an edge between the neighbors of $v$, then the subgraph induced by $v$ and its neighbors has three vertices, three edges and potential at most 2.
This contradicts Lemma~\ref{lem:kahuna}, so the neighbors of $v$ are not adjacent.

Let $G'$ be the assigned graph formed from $G$ by contracting the edge $uv$ into a vertex labeled $uv$, and assign to $uv$ the assignment of $u$.
If $S \subseteq V(G')$ is a nonempty subset with $\p_{G'}(S) \leq 0$, then necessarily $uv \in S$.
Let $S' = (S \setminus \{ uv\}) \cup \{u,v\}$ and observe that $\p_{G}(S') = \p_{G'}(S')$, a contradiction.
Thus $\p_{G'}(S) > 0$ for all nonempty subsets $S \subseteq V(G')$, and since $G' \prec G$, the minimality of $G$ implies $G'$ has an \Ifp{} $\I \sqcup \f$.
Reversing the contraction does not decrease the distance in $G$ between vertices in $G_{\I}$, and $G_{\f}$ remains a forest after adding $v$ to $\f$ since there are no cycles in $G$ that are not in $G'$.
\end{proof}

Before proceeding to Lemma~\ref{lem:move}, note that by Claim~\ref{claim:trivial} a copy of $K_2$ with a vertex in $I$ and the other in $U$, as seen in Figure~\ref{fig:smallconfigs}c, is the only possible induced proper subgraph $H$ of $G$ with a vertex in $U$ that satisfies $\p(H)<3$ and $|V(H)|+|E(H)| \le 4$.

\begin{lem} \label{lem:move}
Let $S$ be a nonempty proper subset of $V(G)$.
If $G'$ is obtained from $G$ by reassigning up to two vertices $u$ and $v$ in $S$ from $U$ to $F$,
then $G'[S]$ has an \Ifp.
\end{lem}
\begin{proof}
We have
\[ |V(G'[S])| + |E(G'[S])| = |V(G[S])| + |E(G[S])| < |V(G)| + |E(G)|, \]
so $G'[S] \prec G$.
Thus $G'[S]$ has an \Ifp{} unless reassigning $u$ and $v$ resulted in $G'$ having some subgraph of non-positive potential.
Suppose $W$ is such a subgraph.
Then $V(W)\cap\{u,v\}\neq\emptyset$ otherwise $\p_G(W)\leq0$, which contradicts the choice of $G$.

As a vertex in $U$ has potential one more than a vertex in $F$, $\p_{G'}(W)\geq\p_G(W)-2$, implying that $\p_G(W)\le 2$ and Lemma~\ref{lem:kahuna} implies $|V(W)|+|E(W)| \leq 4$.
Claim~\ref{claim:trivial} then yields that $W$ can not contain both $u$ and $v$.
Thus $\p_{G'}(W)=\p_G(W)-1$.
However, as $W$ contains a vertex of $U$, Claim~\ref{claim:trivial} also gives that $\p_G(W)\geq2$.
Hence $\p_{G'}(W)\geq1$, a contradiction, so no such $W$ exists.
\end{proof}

The following two claims restrict the local structure around 3-vertices in $\bar{I}$.

\begin{claim} \label{claim:3v2t}
If $v$ is a 3-vertex in $\overline{I}$ with no neighbors in $I$, then $v$ is not incident to a 2-thread in $G$.
\end{claim}
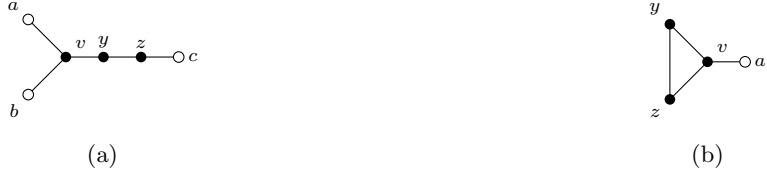
\begin{figure}[!htb]
\centering
	\begin{subfigure}[b]{.48\textwidth}
	\centering
	\begin{tikzpicture}
		\draw (-.5,.5) -- (0,0) -- (-.5,-.5)  (0,0) -- (1.5,0);
		\foreach \x/\y in { 0/0, .5/0, 1/0 } {\fill (\x,\y) circle (2pt); }
		\foreach \x/\y in {-.5/-.5, -.5/.5, 1.5/0 } {\draw[black, fill=white] (\x,\y) circle (2pt); }
		\node[above right] at (0,0) {\scriptsize$v$};
		\node[above left] at (-.5,.5) {\scriptsize$a$};
		\node[below left] at (-.5,-.5) {\scriptsize$b$};
		\node[right] at (1.5,0) {\scriptsize$c$};
		\node[above] at (.5,0) {\scriptsize$y$};
		\node[above] at (1,0) {\scriptsize$z$};
	\end{tikzpicture}
	\caption{}
	\label{fig:32a}
	\end{subfigure}
	\begin{subfigure}[b]{.48\textwidth}
	\centering
	\begin{tikzpicture}
		\draw (0,0) -- (-.5,.5) -- (-.5,-.5) -- (0,0) -- (.5,0);
		\foreach \x/\y in { 0/0, -.5/.5, -.5/-.5 } {\fill (\x,\y) circle (2pt); }
		\draw[black, fill=white] (.5,0) circle (2pt);
		\node[above right] at (0,0) {\scriptsize$v$};
		\node[above left] at (-.5,.5) {\scriptsize$y$};
		\node[below left] at (-.5,-.5) {\scriptsize$z$};
		\node[right] at (.5,0) {\scriptsize$a$};
	\end{tikzpicture}
	\caption{}
	\label{fig:32b}
	\end{subfigure}
\caption{A 3-vertex incident to a 2-thread.}
\label{fig:32t}
\end{figure}

\begin{proof}
Let $v$ be a 3-vertex in $\overline{I}$ with no neighbors in $I$ that is incident to a 2-thread. We consider two cases, depending on whether a 2-thread incident to $v$ is open or closed.

First, suppose that an open 2-thread with vertices $y$ and $z$ is incident to $v$ as in Figure~\ref{fig:32a}; let $a$ and $b$ be the neighbors of $v$ not in this 2-thread.
Let $S = V(G) \setminus \{y,z\}$, and let $G'$ be the assigned graph given by taking $G[S]$ and assigning $a$ and $b$ to $F$, if necessary.
By Lemma~\ref{lem:move}, $G'$ has an \Ifp{} $\I \sqcup \f$.
If $v$ or $c$ is in $\I$, adding $y$ and $z$ to $\f$ extends $\I \sqcup \f$ to $G$ without creating any cycles in $G_{\f}$.
Otherwise $v,c\in \f$ and adding $y$ to $\I$ and $z$ to $\f$ extends $\I \sqcup \f$ to $G$. 

Second, suppose that a closed 2-thread with vertices $y$ and $z$ is incident to $v$ as in Figure~\ref{fig:32b}; let $a$  be the neighbor of $v$ not in this 2-thread.
Let $S = V(G) \setminus \{v,y,z\}$, and let $G'$ be the assigned graph given by taking $G[S]$ and assigning $a$ to $F$, if necessary.
By Lemma~\ref{lem:move}, $G'$ has an \Ifp{} $\I \sqcup \f$.
Adding $v$ and $z$ to $\f$ and $y$ to $\I$ extends $\I \sqcup \f$ from $G'$ to $G$ even in the case that $v \in F$.
\end{proof}

\begin{claim} \label{claim:3}
A 3-vertex in $U$ incident to three 1-threads with bordering vertices in $\overline{I}$ does not appear in $G$.
\end{claim}

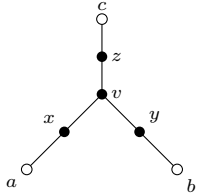
\begin{figure}[!htb]
\centering
	\begin{tikzpicture}
		\foreach \x/\y in { -1/-1, 0/1, 1/-1 } {\draw (0,0) -- (\x,\y); }
		\foreach \x/\y in { 0/0, 0/.5, -.5/-.5, .5/-.5 } {\fill (\x,\y) circle (2pt); }
		\foreach \x/\y in { -1/-1, 0/1, 1/-1 } {\draw[black, fill=white] (\x,\y) circle (2pt); }
		\node[right] at (0,0) {\scriptsize$v$};
		\node[above left] at (-.5,-.5) {\scriptsize$x$};
		\node[above right] at (.5,-.5) {\scriptsize$y$};
		\node[right] at (0,.5) {\scriptsize$z$};
		\node[below left] at (-1,-1) {\scriptsize$a$};
		\node[below right] at (1,-1) {\scriptsize$b$};
		\node[above] at (0,1) {\scriptsize$c$};
	\end{tikzpicture}
\caption{A 3-vertex $v$ incident to three 1-threads.}
\label{fig:3}
\end{figure}
\begin{proof}
Let $v$ be a 3-vertex in $U$ as shown in Figure~\ref{fig:3} where $x,y,z \in U$ are the internal vertices of the 1-threads and $a,b,c \in \overline{I}$ are the other endpoints of the 1-threads.

Suppose first that at most two of $a$, $b$, and $c$ are assigned to $U$; say $c \in F$.
Let $S = V(G) \setminus \{v,x,y,z\}$, and let $G'$ be the assigned graph given by taking $G[S]$ and reassigning $a$ and $b$ to $F$, if necessary.
Lemma~\ref{lem:move} implies there exists an \Ifp{} $\I \sqcup \f$ of $G'$ that extends to an \Ifp{} of $G$ by adding $v$ to $\I$ and $x,y,$ and $z$ to $\f$.

Thus we may assume $a$, $b$, and $c$ are all assigned to $U$ in $G$.
Let $G'=G-\{v,x,y,z\}$ and reassign $a,b$ and $c$ to $F$ in $G'$.
Since $G' \prec G$, $G'$ has an {\Ifp} unless there is some $W\subseteq V(G')$ with $\p_{G'}(W)\leq 0$ and $|W\cap\{a,b,c\}|=3$.
From the reassignment of $a$, $b$, and $c$, $\p_G(W) \leq \p_{G'}(W)+3\cdot(5-4) \leq 3$, which then implies that $\p_G(W \cup \{v,x,y,z\})=\p_G(W)+4\cdot5+6\cdot(-4)\leq -1$, contradicting the choice of $G$.
\end{proof}

Our final claim restricts the structure around 4-vertices in $U$.

\begin{claim}~\label{claim:4U}
A 4-vertex in $U$ incident to three 2-threads and a 1-thread with a bordering vertex in $\overline{I}$ does not appear in $G$.
\end{claim}

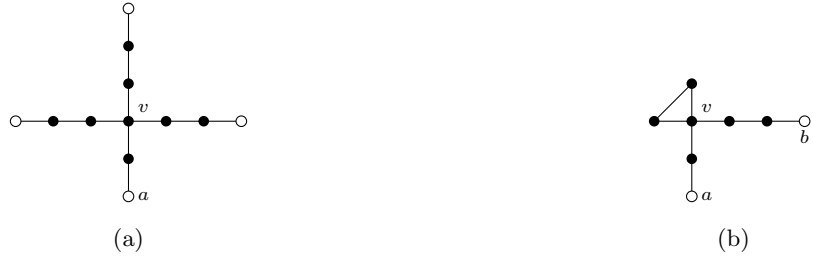
\begin{figure}[!htb]
\centering
	\begin{subfigure}[b]{.48\textwidth}
	\centering
		\begin{tikzpicture}
		\draw (-1.5,0) -- (1.5,0) (0,-1) -- (0,1.5);
		\foreach \x/\y in { 0/0, -1/0, -.5/0, .5/0, 1/0, 0/-.5, 0/.5, 0/1}
			{ \fill (\x,\y) circle (2pt); }
		\foreach \x\y in {0/-1, -1.5/0, 1.5/0, 0/1.5}
			{\draw [black,fill=white] (\x,\y) circle (2pt);}
		\node[above right] at (0,0) {\scriptsize$v$};
		\node[right] at (0,-1) {\scriptsize$a$};
		\end{tikzpicture}
	\caption{}
	\label{fig:4Ua}
	\end{subfigure}
	\begin{subfigure}[b]{.48\textwidth}
	\centering
		\begin{tikzpicture}
		\draw (1.5,0) -- (-.5,0) -- (0,.5) -- (0,-1);
		\foreach \x/\y in { 0/0, -.5/0, .5/0, 1/0, 0/-.5, 0/.5}
			{ \fill (\x,\y) circle (2pt); }
		\foreach\x/\y in {0/-1, 1.5/0} {\draw [black,fill=white] (\x,\y) circle (2pt);}
		\node[above right] at (0,0) {\scriptsize$v$};
		\node[right] at (0,-1) {\scriptsize$a$};
		\node[below] at (1.5,0) {\scriptsize$b$};
		\end{tikzpicture}
	\caption{}
	\label{fig:4Ub}
	\end{subfigure}
\caption{A $4$-vertex $v$ incident to three 2-threads and one 1-thread.}
\label{fig:4U}
\end{figure}
\begin{proof}
Let $v$ be a 4-vertex in $U$ incident with three 2-threads and one 1-thread and let $a \in \overline{I}$ be the other vertex bordering the 1-thread.
Let $T$ be the set of internal vertices in the threads incident to $v$.
At most one of the 2-threads may be closed, as depicted in Figure~\ref{fig:4U}.

Let $S = V(G)\setminus(T\cup\{v\})$, and let $G'$ be the assigned graph formed by taking $G[S]$ and assigning $a$ to $F$, if necessary.
Lemma~\ref{lem:move} implies that there exists an \Ifp{} $\I \sqcup \f$ of $G'$, necessarily with $a\in \f$.
Adding $v$ to $\I$ and the vertices of $T$ to $\f$ extends $\I \sqcup \f$ to $G$ so that vertices in $\I$ have pairwise distance at least two in $G$ and $G_{\f}$ is a forest.
Notice that since any cycle in $G_{\f}$ that would be created must use $v$, $G_{\f}$ is a forest.
\end{proof}

\section{Proof of Theorem~\ref{thm:pot}}\label{sec:main} 

We use discharging the prove the following lemma.

\begin{lem} \label{lem:discharging}
If $G$ is a graph satisfying Claims~\ref{claim:prevredconfig},~\ref{claim:2-F} and~\ref{claim:3v2t}--\ref{claim:4U}, then $\p_G(G) \le 0$.
\end{lem}

As we previously demonstrated that the minimal counterexample to Theorem~\ref{thm:pot} satisfies Claims~\ref{claim:prevredconfig},~\ref{claim:2-F} and~\ref{claim:3v2t}--\ref{claim:4U}, this demonstrates a contradiction and no counterexample exists.

\begin{proof}[Proof of Lemma~\ref{lem:discharging}]
Suppose that $G$ satisfies  Claims~\ref{claim:prevredconfig},~\ref{claim:2-F} and~\ref{claim:3v2t}--\ref{claim:4U}.
Assign an initial charge $\ic$ to vertices of $G$ as follows:
\[ \ic(v) = 2d(v) -
	\begin{cases}
	1, & v\in I \\
	4, & v\in F \\
	5, & v\in U.
	\end{cases} \]
Observe that $\sum\limits_{v\in V(G)} \ic(v) = 4|E(G)|-|I|-4|F|-5|U|=-\p(G)$.

We distribute charge using three rules, (R1), (R2), and (R3), in order, and use $\mu_i(v)$ to denote the charge on a vertex after rule $i$ for $i \in \{1,2,3\}$.
\begin{enumerate}[(R1)]
\item If $v \in V(G)$ satisfies $\ic(v) \geq d(v)$, then $v$ sends charge 1 to each neighbor $u \in N(v)$.
\item If $v$ is the internal vertex of a 1-thread and $\mu_1(v) < 0$, then $v$ pulls charge $\frac{1}{2}$ from each of its neighbors.
\item If $v$ is an internal vertex of a 2-thread and $\mu_2(v) < 0$, then $v$ pulls charge $1$ from its neighbor on the border of the thread.
\end{enumerate}

We will demonstrate $\fc(v) \geq 0$ for all $v \in V(G)$, , which implies that
\[-\p(G) = \sum\limits_{v\in V(G)} \ic(v) = \sum\limits_{v\in V(G)} \fc(v)\geq 0.\]
Observe that if $u \in V(G)$ is not internal to any thread and $v \in N(u)$, then $u$ sends charge to $v$ by at most one rule (R1), (R2), or (R3).

If $v$ is a vertex in $I$, a vertex in $F$ with $d(v)\geq 4$, or a vertex in $U$ with $d(v) \geq 5$, then $\ic(v)\geq d(v)$.
Consequently, $v$ has $\mu_1(v)=\mu_2(v)=\mu_3(v) \geq 0$ since $v$ sends charge $d(v)$ during (R1) and does not send charge using (R2) or (R3).

Next, suppose $v$ is a vertex in $F$ with $d(v) \le 3$.
By Claim~\ref{claim:2-F}, $d(v) \geq 3$. 
If $v$ is incident to a 2-thread, then by Claim~\ref{claim:3v2t}, $v$ has a neighbor in $I$.
This neighbor sends charge 1 to $v$ by (R1) and $v$ sends charge at most 1 to each other neighbor by (R2) or (R3), so $\fc(v) \geq \ic(v)+1-2 \ge 0$.
Otherwise, from (R2), $v$ sends charge at most $\frac{3}{2}$ to incident 1-threads and $\fc(v)\geq\ic(v)-3\cdot\frac{1}{2}=\frac{1}{2}>0$, as desired.

Recall that by Claim~\ref{claim:prevredconfig}, no vertex in $U$ has degree less than two.
Since $G$ contains no $3^+$-thread by Claim~\ref{claim:prevredconfig}, if $v$ is a $2$-vertex in $U$, then $\mu_0(v) = -1$ and $\mu_3(v) \geq 0$ by either (R1), (R2), or (R3).

Assume then that $v$ is a vertex in $U$ with $d(v) = 3$.
If $v$ is incident to a 2-thread, then by Claim~\ref{claim:3v2t}, $v$ has a neighbor in $I$.
This neighbor sends charge 1 to $v$ by (R1) and $v$ sends charge at most 1 to each other neighbor by (R2) or (R3), so $\fc(v) \geq \ic(v)+1-2 \ge 0$, as desired.
If $v$ is not incident to any 2-threads and is incident to fewer than three 1-threads, then $\fc(v)\geq\ic(v)-2\cdot \frac{1}{2} \geq0$, as desired.
Otherwise, $v$ is incident to exactly three 1-threads, and at least one of the 1-threads is bordered by a vertex $a$ in $I$ by Claim~\ref{claim:3}.
Since $a$ sends charge 1 to the internal vertex of the 1-thread by (R1), $v$ sends charge at most $\frac{1}{2}$ to the other neighbors by (R2), and hence $\fc(v)\geq\ic(v)-2\cdot \frac{1}{2}\geq0$, as desired.

Finally, suppose $v$ is a vertex in $U$ with $d(v) = 4$.
By Claim~\ref{claim:prevredconfig}, $v$ is not incident to four 2-threads.
By Claim~\ref{claim:4U}, if $v$ is incident to three 2-threads and a 1-thread, then the other vertex $a$ bordering the 1-thread is in $I$.
Since $a$ sends charge 1 to the 1-thread by (R1), $v$ sends charge at most 1 to at most three neighbors by (R2) and (R3) and hence $\fc(v)\geq\ic(v)-3\left(1\right)\geq0$, as desired.
If $v$ is incident to three 2-threads and no other thread, then $\fc(v) \ge \ic(v) - 3\cdot 1 \ge 0$.
Finally, if $v$ is incident to at most two 2-threads and up to two 1-threads, then $\fc(v) \ge \ic(v) - 2\cdot 1 - 2\cdot \frac{1}{2} \geq0$.

Therefore, every vertex in $G$ has non-negative final charge, which completes the proof.
\end{proof}

\bibliographystyle{abbrv}
\bibliography{Bib}

\end{document}